\documentclass{article}

\usepackage{graphicx,latexsym,amsmath,amsthm,amssymb,color,verbatim,mathrsfs,
bbm,amscd}
\usepackage[latin1] {inputenc}
\usepackage{youngtab,url}
\usepackage[all]{xy}

\newtheorem{theorem}[equation]{Theorem}
\newtheorem{corollary}[equation]{Corollary}

\newtheorem{proposition}[equation]{Proposition}
\newtheorem{conjecture}[equation]{Conjecture}
\newtheorem{claim}[equation]{Claim}

\theoremstyle{definition}
\newtheorem{remark}[equation]{Remark}
\newtheorem{definition}[equation]{Definition}

\newcommand{\tensor}{\smash{\textstyle\bigotimes}}
\newcommand{\Sym}{\mathsf{Sym}}
\newcommand{\GL}{\mathsf{GL}}
\newcommand{\FH}{\ensuremath{\Psi}}
\newcommand{\Ch}{\mathscr{C}}

\newcommand{\la}{\lambda}
\newcommand{\HWV}{\mathsf{HWV}}
\newcommand{\IC}{\ensuremath{\mathbb{C}}}
\newcommand{\IN}{\ensuremath{\mathbb{N}}}
\newcommand{\IZ}{\ensuremath{\mathbb{Z}}}
\newcommand{\aS}{\ensuremath{\mathfrak{S}}}

\newcommand{\mult}{\text{mult}}

\newcommand{\inj}{\ar@{^{(}->}}
\newcommand{\surj}{\ar@{->>}}

\newcommand{\diag}{\text{diag}}

\title{Symmetrizing Tableaux and the 5th case of the Foulkes Conjecture}
\author{Man-Wai (Mandy) Cheung${}^1$, Christian Ikenmeyer${}^2$, Sevak Mkrtchyan${}^3$}
\date{\today}

\begin{document}
\maketitle

\begin{abstract}
The Foulkes conjecture states that the multiplicities in the plethysm $\Sym^a(\Sym^b V)$
are at most as large as the multiplicities in the plethysm $\Sym^b(\Sym^a V)$ for all $a \leq b$.
This conjecture has been known to be true for $a \leq 4$.
The main result of this paper is its verification for $a = 5$.
This is achieved by performing a combinatorial calculation on a computer and
using a propagation theorem of Tom McKay from 2008.

Moreover, we obtain a complete representation theoretic decomposition of the vanishing ideal of the $5$th Chow variety in degree~$5$,
we show that there are no degree 5 equations for the 6th Chow variety,
and we also find some representation theoretic degree~$6$ equations for the $6$th Chow variety.
\end{abstract}

{\mbox{~}}

{\noindent\footnotesize Keywords: Representation theory of the symmetric group; Plethysm; Foulkes Conjecture; Foulkes-Howe Conjecture, Chow variety

}

{\mbox{~}}

\noindent{\footnotesize MSC2010: 20C30, 20C40, 14-04

}

\footnotetext[1]{University of California, San Diego, mwc31$@$cam.ac.uk}
\footnotetext[2]{Texas A\&M University, ciken$@$math.tamu.edu}
\footnotetext[3]{University of Rochester, sevak.mkrtchyan$@$rochester.edu}

\section{Introduction}
The Foulkes conjecture~\cite{Fou:50} states an inequality of certain representation theoretic multiplicities.
For a finite dimensional vector space $V$, let $\Sym^b V$ denote its $b$th symmetric power and let
$\Sym^a\Sym^b V$ denote the $a$th symmetric power of $\Sym^b V$.
We will always assume that $\dim V$ is large enough, i.e., $\dim V \geq \max\{a,b\}$.
As a $\GL(V)$ representation, $\Sym^a\Sym^b V$ has an intriguing structure with interesting long-standing open questions like the following one.
\begin{conjecture}[Foulkes conjecture]\label{conj:foulkes}
Let $a, b \in \IN_{>0}$ with $a \leq b$.
Then for every partition $\lambda$
the multiplicity of the irreducible $\GL(V)$ representation $\{\lambda\}$ in the plethysm
$\Sym^a \Sym^b V$ is at most as large as 
the multiplicity of $\{\lambda\}$ in $\Sym^b \Sym^a V$.
\end{conjecture}
The inequality $a \leq b$ is important: We know that $\Sym^a \Sym^b V$
contains irreducible $\GL(V)$ representations $\{\la\}$ for which the Young diagram of $\la$ has up to $a$ rows,
but $\Sym^b \Sym^a V$ contains irreducible $\GL(V)$ representations $\{\la\}$ for which the Young diagram of $\la$ has up to $b$ rows.

Brion \cite{Bri:93} showed that Conjecture~\ref{conj:foulkes} is true in the cases where $b$ is large enough with respect to $a$.
Conjecture~\ref{conj:foulkes} is true if we only consider partitions $\la$ with at most 2 rows~\cite{Her:1854}, a phenomenon called Hermite reciprocity.
Manivel showed that Conjecture~\ref{conj:foulkes} is also true for all partitions $\la$ with very long first rows \cite[Thm.~4.3.1]{mani:98}.
In this case we do not only have an inequality of the multiplicities, but equality.

Conjecture~\ref{conj:foulkes} is known to be true for all $a \leq 4$:
For $a \leq 2$ see the explicit formulas in \cite{Thr:42};
for the case $a \leq 3$ see \cite{DS:00}.
The proof for $a=4$ uses a canonical map $\Psi_{4,4}$ and a propagation theorem by McKay.
This study is related to a stronger conjecture, sometimes called the Foulkes-Howe conjecture, which at least dates back to Hadamard \cite{Had:97}.
We will explain this in more detail now.
Clearly Conjecture~\ref{conj:foulkes} is equivalent to saying that there exists a $\GL(V)$ equivariant inclusion map
\begin{equation}\label{eq:foulkesconj}
\Sym^a\Sym^b V \hookrightarrow \Sym^b\Sym^a V.
\end{equation}
A natural candidate for the map in \eqref{eq:foulkesconj} was the following map $\Psi_{a,b}$:
\[
\xymatrix{
\Sym^a \Sym^b V \inj[d]^\iota \ar[r]^{\Psi_{a,b}} & \Sym^b \Sym^a V \\
\tensor^a\tensor^b V \ar[r]^{r} & \tensor^b \tensor^a V \surj[u]_{\varrho}
}
\]
where $\iota$ denotes the canonical embedding of symmetric tensors in the space of all tensors,
$\varrho$ is the canonical projection,
and $r$ is the canonical isomorphism given by reordering tensor factors.

Hadamard conjectured \cite{Had:97} that $\Psi_{a,b}$ is injective for all $a \leq b$.
Howe \cite[p.~93]{How:87} wrote that this ``is reasonable to expect''.
This conjecture is known as the Foulkes-Howe conjecture.
However, M\"uller and Neunh\"offer \cite{MN:05} showed that $\Psi_{5,5}$ has a nontrivial kernel. As is in our case, the results in \cite{MN:05} were also based on computer computations.

McKay \cite{McK:08} contributed the following important propagation theorem.
\begin{theorem}\label{thm:main}
If $\Psi_{a,(b-1)}$ is injective, then $\Psi_{a,c}$ is injective for all $c \geq b$.
\end{theorem}

Theorem~\ref{thm:main} allows us to verify Conjecture~\ref{conj:foulkes} in infinitely many cases while only doing a finite calculation:
\cite{MN:05} calculate that $\Psi_{4 \times 4}$ is injective, so Theorem~\ref{thm:main} implies the following corollary.
\begin{corollary}\label{cor:computation}
Conjecture~\ref{conj:foulkes} is true for $a = 4$.
\end{corollary}

Our main result is the following theorem.
\begin{theorem}\label{cor:bettercomputation}
Conjecture~\ref{conj:foulkes} is true for $a = 5$.
\end{theorem}

We obtain this result as a consequence of the first part of the following theorem.
\begin{theorem}\label{thm:maininj}

\begin{enumerate}
\item[(a)] $\FH_{5,6}$ is injective.
\item[(b)]
As a $\GL(V)$-representation the kernel of $\FH_{5,5}$ decomposes as
\begin{eqnarray*}
\ker \FH_{5,5} &=& 
\{(14, 7, 2, 2)\} \oplus \{(13, 7, 2, 2, 1)\} \oplus \{(12, 7, 3, 2, 1)\} \oplus \{(12, 6, 3, 2, 2)\} \\
&\oplus& \{(12, 5, 4, 3, 1)\} \oplus \{(11, 5, 4, 4, 1)\} \oplus \{(10, 8, 4, 2, 1)\} \oplus \{(9, 7, 6, 3)\}.
\end{eqnarray*}

\item[(c)]
Let $S$ denote the set of partitions $\la$ that satisfy
\begin{itemize}
 \item $\la$ has 36 boxes,
 \item $\la$ has 6 rows (i.e., $\la_6\neq 0$) and
 \item $\la$ appears with multiplicity exactly 1 in $\Sym^6\Sym^6 V$.
\end{itemize}
The following types appear in $\ker\FH_{6,6}$ (of course with multiplicity 1), and no other
partition from $S$ appears in $\ker\FH_{6,6}$:
\begin{eqnarray*}
&&(20,7,6,1,1,1),
(9,8,8,7,3,1),
(16,7,5,3,3,2),
(15,9,4,3,3,2),\\
&&(11,6,6,6,4,3),
(11,9,7,3,3,3).
\end{eqnarray*}

\end{enumerate}
\end{theorem}

Although $\Psi_{5,5}$ is not injective, if $a=b$, then $\Sym^a \Sym^b V = \Sym^b \Sym^a V$.
Therefore Theorem~\ref{thm:main} and Theorem~\ref{thm:maininj}(a) together imply Theorem~\ref{cor:bettercomputation}.

\begin{remark}
Regarding $\Sym^b V$ as polynomials over $V$, let $\Ch_b:=\overline{\GL_V (X_1 X_2 \cdots X_b)} \subseteq \Sym^b V$ denote the $\GL_V$ orbit closure of the monomial $X_1 X_2 \cdots X_b$.
We call this orbit closure the $b$th Chow variety.
The kernel of $\Psi_{a,b}$ is the homogeneous degree $a$ part of the vanishing ideal of $\Ch_b$,
as was shown by Hadamard (see e.g.~\cite[Section~8.6]{Lan:11}).
This interpretation is what greatly inspires the algorithm described in this paper, which is used to obtain the results in Theorem~\ref{thm:maininj}.

Moreover, we remark that $\Ch_b$ is an important subvariety of the determinant orbit closure
in geometric complexity theory \cite{Lan:11}, \cite{Kum:12}, \cite{BHI:15}:
an approach towards resolving major open questions in theoretical computer science
via geometry and representation theory.

Theorem~\ref{thm:maininj} can be interpreted as follows: There are no equations for $\Ch_6$ up to degree 5,
the vanishing ideal of $\Ch_5$ in degree 5 is identified as a $\GL_V$-representation,
and some parts of the vanishing ideal of $\Ch_6$ in degree 6 are determined. \hfill\scalebox{0.8}{$\blacksquare$}\end{remark}

For more information on the history of the Foulkes conjecture we refer the interested reader to \cite[Section~7.1]{Lan:15}.

We note that the result in Theorem~\ref{thm:maininj}(b)
uncovers the precise structure of $\ker\Psi_{5,5}$, which is slightly different from the result
``with high probability'' in \cite[(6.1.4)]{ike:12b}.

\subsubsection*{Organization of the paper}
As stated above, Theorem~\ref{cor:bettercomputation} follows from McKay's theorem and Theorem~\ref{thm:maininj}. The rest of the paper is devoted to the the algorithm behind the computations in Theorem~\ref{thm:maininj}. In Section~\ref{sec:reducetohwv} we use representation theoretic considerations to reduce the problem of computing $\ker\Psi_{a,b}$ to computing the dimension of the kernel of the restriction of $\Psi_{a,b}$ to the subspace of $\Sym^a\Sym^b V$ of highest weight vectors of type $\lambda$ for all $\lambda$. In Section \ref{sec:naive} we give a detailed description and a justification of the algorithm to compute this dimension. An overview is given in Section~\ref{subsec:overview}. While the algorithm includes various mathematical optimizations, the problem is still computationally too big to be feasible. In Section~\ref{sec:speedup} we describe various techniques we used to speed up the algorithm and make the calculation feasible. In Section~\ref{sec:implandrun} we give remarks on the implementation and running time of the algorithm.

\subsection*{Acknowledgments}
We thank JM Landsberg for bringing this problem to our attention and for valuable discussions.
This work was initiated at the 2012 Mathematics Research Communities \emph{Geometry and Representation Theory Related to Geometric Complexity and Other Variants of P v. NP} in Snowbird, Utah. We thank Texas A\&M University for its hospitality for a one week collaboration effort funded by an AMS MRC Collaboration grant.

\section{Reduction to highest weight vectors}
\label{sec:reducetohwv}
Our goal is to compute the kernel of the $\GL(V)$-morphism $\FH_{a,b}$ and decompose it as a $\GL(V)$-representation.
Using Schur's lemma we can treat each $\GL(V)$-isomorphism type $\lambda$ independently
and calculate the kernel of the restriction $\FH_{a,b}^{\la\text{-iso}}$ of $\FH_{a,b}$ to the $\la$-isotypic component of $\Sym^a\Sym^b V$:
\[
\ker \FH_{a,b} = \bigoplus_{\la} \ker\FH_{a,b}^{\la\text{-iso}},
\]
where the sum is over all partitions $\la$ with $|\la|:=\sum_{i} \la_i = ab$.
The following definition simplifies our computation.
\begin{definition}\label{def:hwv}
 Let $\mathcal U \subseteq \GL(V)$ denote the group of upper triangular matrices with 1s on the main diagonal.
 We denote by $\diag(\alpha_1,\ldots,\alpha_{\dim V})$ the diagonal matrix with complex numbers $\alpha_1,\ldots,\alpha_{\dim V}$ on the main diagonal.
 Let $W$ be a polynomial $\GL(V)$-representation.
 Let $\la$ be a partition.
 A vector $f \in W$ that is invariant under the action of $\mathcal U$ and that satisfies
 \[
 \diag(\alpha_1,\ldots,\alpha_{\dim V}) \cdot f = \alpha_1^{\la_1} \cdots \alpha_{\dim V}^{\la_{\dim V}} f
 \]
 is called a \emph{highest weight vector of type $\la$}.\hfill\scalebox{0.8}{$\blacksquare$}\end{definition}

Let $W := \Sym^a\Sym^b V$ and let $W':=\Sym^b\Sym^a V$.
Let $\HWV_\la(W)$ denote the vector space of highest weight vectors of type $\la$ in~$W$.
Since every irreducible $\GL(V)$-representation has a unique highest weight vector (up to scale),
we only have to check the kernel of the restriction $\FH_{a,b}^\la:\HWV_\la(W)\to\HWV_\la(W')$ of $\FH_{a,b}^{\la\text{-iso}}$ to $\HWV_\la(W)$:
\[
\ker \FH_{a,b}^{\la\text{-iso}} = \text{linspan} \{g f \mid g \in \GL(V), f \in \ker \FH_{a,b}^\la\}.
\]
So $\ker \FH_{a,b}^{\la\text{-iso}}$ is a $\la$-isotypic $\GL(V)$ representation
that contains the irreducible representation of type $\la$
with a multiplicity of $\dim\ker \FH_{a,b}^\la$.
We write:
\[
\mult_\la(\ker \FH_{a,b}) = \dim\ker \FH_{a,b}^\la.
\]
Hence we are left with determining $\dim\ker \FH_{a,b}^\la$.

Let $p := \dim\HWV_\la(W)$ and $p':= \dim\HWV_\la(W')$. In our cases we have $p\leq p'$ (otherwise this would disprove Conjecture~\ref{conj:foulkes}).
To compute $\dim\ker \FH_{a,b}^\la$ we first determine a basis $\{f_i\}_{1 \leq i \leq p}$ of $\HWV_\la(W)$
and a basis $\{f'_i\}_{1 \leq i \leq p'}$ of $\HWV_\la(W')$.
Using these bases we can write down the $p' \times p$ transformation matrix of $\FH_{a,b}^\la$ and determine its kernel with linear algebra.

\section{The algorithm}\label{sec:naive}
\subsection{Overview}\label{subsec:overview}
In our cases for every $\la$ the natural numbers $p$ and $p'$ can be readily computed using the computer programs \textsf{Schur}\footnote{online available at \url{http://sourceforge.net/projects/schur}} or \textsf{LiE}\footnote{online available at \url{http://www-math.univ-poitiers.fr/~maavl/LiE}}.
We then construct a basis $\{f_1,\ldots,f_p\}$ of $\HWV_\la(W)$ by
choosing more and more ``random'' elements in $\HWV_\la(W)$ and by checking their linear independence until we have $p$ many linearly independent ones.
In the same way we construct a basis $\{f'_1,\ldots,f'_{p'}\}$ of $\HWV_\la(W')$,
but this time we construct it together with a set of $p'$ distinct points $\{v'_1,\ldots,v'_{p'}\}, v'_j \in \Sym^a V$,
such that the $p' \times p'$ matrix $\Big(f'_i(v'_j)\Big)_{i,j}$ of function evaluations has full rank.
This means that every element $f'$ of $\HWV_\la(W')$ is uniquely defined by its evaluation list $(f'(v'_1),\ldots,f'(v'_{p'}))$. Note that here the $f_i\in \Sym^b\Sym^a V$ are regarded as functions on $\Sym^a V$ as opposed to vectors from the tensor product space.

Then we apply $\FH_{a,b}$ to $f_1,\ldots,f_p$ and determine the rank of the $p' \times p$ matrix
given by the function evaluations $\Big(\FH_{a,b}(f_i)(v'_j)\Big)_{i,j}$.
The kernel of $\FH_{a,b}^\la$ is given precisely by the kernel of this matrix.
In other words a linear combination $\sum_{i=1}^p \alpha_i \FH_{a,b}(f_i)$ vanishes identically iff
$\sum_{i=1}^p \alpha_i \FH(f_i)(v'_j)=0$ for all $j$.

\subsection{Numerical considerations}
Even though at a first glance the method described in Section~\ref{subsec:overview} might look like it works only ``with high probability''
or ``numerically'', it gives \emph{symbolic} results.
The key property that we use is that the $p' \times p'$ matrix $\Big(f'_i(v'_j)\Big)_{i,j}$ has full rank by construction,
so that a function in $\HWV_\la(W')$ is uniquely determined by its evaluations at the points $v'_1,\ldots,v'_{p'}$.

We will construct the functions $f_i$, $f'_i$ over the integers and we will also do so for the points $v'_j$.
Therefore all evaluation results are integers.
Integer calculations can be handled without any numerical errors by software with bignum arithmetic support.

\subsection{Constructing the highest weight vector bases}\label{subsec:tableaux}
In this section we describe how to obtain the functions $f_i$, $1 \leq i \leq p$ and $f_i'$, $1 \leq i \leq p'$.
Since their number of monomials is very large,
we will not compute the monomials, but we will provide a method to evaluate the functions without writing them down.
This is the same phenomenon as for the classical $n \times n$ determinant,
which has $n!$ summands but which can be computed in polynomial time.

\subsubsection{A basis for the tensor power}

A \emph{Young diagram} is a left- and top-aligned array of boxes.
We identify partitions with Young diagrams by interpreting the entries in the partition as row lengths.
For example, the Young diagram corresponding to the partition $(4,2)$ is
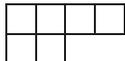
\[
\yng(4,2).
\]
A \emph{Young tableau} is a Young diagram in which we write a natural number in each box.
Different boxes can have the same entry.
A Young tableau is called \emph{semistandard} if
the numbers increase from top to bottom in each column and never decrease from left to right in each row.
An example of a semistandard Young tableau is the following:
\begin{equation}\label{eq:semistd}
\Yvcentermath1\young(1133,22).
\end{equation}
If a Young tableau $T$ with $ab$ boxes has exactly $b$ 1s, exactly $b$ 2s, $\ldots$, and exactly $b$ $a$s,
then we say that $T$ has \emph{rectangular content $a \times b$}.
The example above has rectangular content $3 \times 2$.
A special case are the Young tableaux with $d$ boxes and rectangular content $d \times 1$:
Each number from $1,\ldots,d$ appears exactly once.
A semistandard Young tableau with rectangular content $d \times 1$ is called a \emph{standard tableau}.

We want to construct a basis of $\HWV_\la(\Sym^a \Sym^b V)$.
There is a nice basis of $\HWV_\la(\tensor^a \Sym^b V)$ 
that is obtained from a nice basis of $\HWV_\la(\tensor^{ab} V)$ by a process we call \emph{inner symmetrization}.
We describe this basis of $\HWV_\la(\tensor^{ab} V)$ now.
Let $\{e_i\}_{1 \leq i \leq \dim V}$ be the standard basis of $V$.
We start with a canonical construction of a highest weight vector of type $\la$ in $\tensor^{ab}V$:
For each column of length $k$ in $\lambda$ we take the wedge product $e_1 \wedge e_2 \wedge \cdots \wedge e_k$
and tensor them together from left to right to obtain the tensor $\zeta_\la \in \tensor^{ab}V$.
For example, for the partition $(4,2)$ we get
\begin{equation}\label{eq:def:fla}
\zeta_{(4,2)} := (e_1 \wedge e_2) \otimes (e_1 \wedge e_2) \otimes e_1 \otimes e_1
\end{equation}
and for the partition $(7,3,2)$ we get
\begin{equation*}
\zeta_{(7,3,2)} := (e_1 \wedge e_2 \wedge e_3) \otimes (e_1 \wedge e_2 \wedge e_3) \otimes (e_1 \wedge e_2) \otimes e_1 \otimes e_1 \otimes e_1 \otimes e_1.
\end{equation*}
The vector $\zeta_\la$ is a highest weight vector of type $\la$ in $\tensor^{ab}V$, which is easy to verify (recall Definition~\ref{def:hwv}).
The group $\aS_{ab}$ acts on $\tensor^{ab}V$ by permuting the tensor factors.
Its action commutes with the action of $\GL(V)$.
Therefore, for all $\pi \in \aS_{ab}$ we have that $\pi \zeta_\la$ is a highest weight vector of type $\la$.
Indeed, Schur-Weyl duality (see e.g.\ \cite[Ch.~9, eq.~(3.1.4)]{Pro:07} or \cite[eq.~(9.1)]{gw:09}) says that the set $\{\pi \zeta_\la \mid \pi \in \aS_{ab}\}$
is a generating set for $\HWV_\la(\tensor^{ab}V)$.

Let $T_\la$ denote the \emph{column-standard tableau of shape $\la$}, i.e., the standard tableau of shape $\la$ whose entries are consecutive in each column.
For example, for $\la=(4,2)$ we have
\[
\Yvcentermath1 T_{(4,2)} = \young(1356,24).
\]
The group $\aS_{ab}$ acts on the set of tableaux with rectangular content $(ab)\times 1$ by sending the entries $i$ to $\pi^{-1}(i)$.
We can identify $\pi T_\la$ with $\pi\zeta_\la$. Indeed, the numbers in the columns of $\pi T_\la$ are exactly the indices of the position that are wedged with each other in $\pi \zeta_\la$.
Since $\zeta_\la$ is a tensor product of special wedge products,
the linear relations among all $\pi \zeta_\la$ are given by
the linear relations among the principal minors of generic matrices:
The Grassmann and Pl\"ucker relations.
We can write these relations in terms of tableaux as follows:
\begin{itemize}
\item $T = -T'$ if $T$ and $T'$ differ in a single column by a single transposition. This is the Grassmann relation.
\item $T = \sum_{S} S$,
 where the sum is over all tableaux~$S$ that arise from~$T$ 
 by exchanging for some $j$ and $k$ the top $k$ elements from the $(j+1)$th column with any selection of 
 $k$~elements in the $j$th column, preserving their vertical order, see \cite[p.~110]{fult:97}.
 These are known as the Pl\"ucker relations.
\end{itemize}
For example, from the Pl\"ucker relations it follows that if $T$ and $T'$
arise from each other by exchanging two columns of the same length, then $T=T'$.
We will use this in Section~\ref{subsec:usingsymmetries}.

The Grassmann-Pl\"ucker relations can be used to identify a set of tableaux that form a basis of
$\HWV_\la(\tensor^{ab}V)$: Those are the standard tableaux, see \cite[p.~110]{fult:97}.
The algorithm that expresses a Young tableau with rectangular content $(ab) \times 1$ as a linear combination
over the set of standard tableaux is called \emph{straightening}.

\subsubsection{The inner and outer symmetrizations}

We now consider the canonical projection map
$s : \tensor^{ab} V \twoheadrightarrow \tensor^a\Sym^bV$,
which can be written as an element in the group algebra $\IC[\aS_{ab}]$:
\[
s := \frac 1{(b!)^a} \sum_{ \substack{ {\sigma_1 \in \aS_{\{1,2,\ldots,b\}}} \\ {\sigma_2 \in \aS_{\{b+1,b+2,\ldots,2b\}}} \\ \vdots \\ {\sigma_a \in \aS_{\{b(a-1)+1,b(a-1)+2,\ldots,ab\}}}}  }\sigma_1 \sigma_2 \cdots \sigma_a.
\]
We call $s$ the \emph{inner symmetrization map}.
A tableau $T$ is mapped to a sum of $(b!)^a$ tableaux, as illustrated in the following example where $a=3$ and $b=2$.
\begin{eqnarray*}
\Yvcentermath1s\bigg(\young(1256,34)\bigg) =  && \frac 1 8 \bigg(
\Yvcentermath1\young(1256,34) + \young(2156,34) + \young(1256,43) + \young(2156,43) \\
&&+ \Yvcentermath1\young(1265,34) + \young(2165,34) + \young(1265,43) + \young(2165,43)\bigg).
\end{eqnarray*}
Here the pairs $(1,2)$, $(3,4)$, and $(5,6)$ are symmetrized independently.

The cases in which $s(T)=0$ are those where positions are symmetrized that lie in the same column.
In all other cases it can be shown (for example by tensor contraction) that $s(T)\neq 0$ and we use a short notation for the resulting sum of tableaux:
In each summand, replace the first $a$ entries $1,\ldots,a$ with 1,
the second $a$ entries $a+1,\ldots,2a$ with 2, and so on.
Of course this gives the same tableau for each summand, so for example in the above example we obtain the tableau in \eqref{eq:semistd} for each of the 8 summands.
This tableau with rectangular content $a \times b$ shall now represent the sum $s(T)$.
If $T$ is standard and $s(T)\neq 0$, then $s(T)$ is semistandard.
The semistandard tableaux with rectangular content $a \times b$ form a basis of $\HWV_\la(\tensor^a\Sym^bV)$.
It is clear that on the space of tableaux with rectangular content $a\times b$ the Grassmann-Pl\"ucker relations hold and that the straightening algorithm can be used to express a Young tableau with rectangular content $a \times b$
over the set of semistandard tableaux with rectangular content $a \times b$.

There is no such nice basis known for $\HWV_\la(\Sym^a\Sym^bV)$.
This is an important open problem in representation theory:
Find a combinatorial description of $\dim\HWV_\la(\Sym^a\Sym^bV)$, which is called the \emph{plethysm coefficient}.
We only know that \[\left\{\sum_{\sigma \in \aS_a}\sigma T \,\middle|\, T \text{ semistandard of type $\la$ with rectangular content $a \times b$}\right\}\]
is a generating set of $\HWV_\la(\Sym^a\Sym^bV)$.
Since in our cases we have $a \leq 6$ and hence $|\aS_a|\leq 6!=720$, this is not a major computational problem.

For a given Young tableau $T$ with rectangular content $a \times b$ let us call $\varrho(T) := \frac 1 {a!} \sum_{\sigma \in \aS_a}\sigma T$
an \emph{$(a,b)$-symmetrized tableau}. We will often write $f = \varrho(T)$.
So the set of all $(a,b)$-symmetrized tableaux $f$ of shape $\la$ is a generating set of $\HWV_\la(\Sym^a\Sym^bV)$.
An $(a,b)$-symmetrized tableau shall be depicted by replacing the numbers by letters,
thinking of the set of letters as an unordered set. So, for example,
\[
\Yvcentermath1\varrho\bigg(\young(1133,22)\bigg) = \young(AACC,BB).
\]

\subsection{Applying the map}
Using the tableau descriptions from \ref{subsec:tableaux} we can describe the map
\[
\FH_{a,b}^\la : \HWV_\la(\Sym^a\Sym^bV) \to \HWV_\la(\Sym^b\Sym^aV)
\]
in terms of tableaux.
A generating set of the left hand side is given by the set of all $(a,b)$-symmetrized tableaux of shape $\la$,
while a generating set of the right hand side is given by the set of all $(b,a)$-symmetrized tableaux of shape $\la$.

\begin{proposition}[A combinatorial description of $\Psi_{a,b}$]\label{pro:combinat}
 Given an $(a,b)$-symmetrized tableaux $f$ of shape $\la$, $\FH^\la_{a,b}(f)$ equals the sum over all tableaux $S$ that are obtained
 by putting the numbers $1,\ldots,b$ on the positions of the A's in $T$ in any order,
and putting the numbers $1,\ldots,b$ on the positions of the B's in $T$ in any order,
and so on.
\end{proposition}
We highlight this combinatorial definition with an example.
\begin{eqnarray*}
\Yvcentermath1\FH_{3,2}\bigg(\young(AACC,BB)\bigg) &=& \Yvcentermath1 \young(1212,12) + \young(2112,12) + \young(1221,12) + \young(2121,12) \\
&+& \Yvcentermath1 \young(1212,21) + \young(2112,21) + \young(1221,21) + \young(2121,21) \\
&=& \Yvcentermath1 -4 \, \young(1112,22),
\end{eqnarray*}
where the last equality is the result of applying the straightening algorithm.
Note that the map in Proposition~\ref{pro:combinat} is well-defined and that by
grouping the summands correctly we can see that the image $\FH_{a,b}(f)$ is a sum of $(b,a)$-symmetrized tableaux.

\begin{proof}[Proof of Propostion~\ref{pro:combinat}]
For an $(a,b)$-symmetrized tableau $f = \sum_{\sigma \in \aS_a}T$ we take a summand $T$ with rectangular content $a \times b$
and interpret it according to its definition as a sum of tableaux with rectangular content $(ab) \times 1$.
Now we apply the map $r$,
which means that we apply the corresponding permutation from $\aS_{ab}$ to the tableaux.
We then turn summands into
tableaux with rectangular content $b \times a$ by applying the inner symmetrization map $\tensor^{ab}V \to \tensor^b \Sym^a V$.
Upon close inspection one sees that we do not need to apply the outer symmetrization, as the resulting sum is already invariant under the action of $\aS_b$.
\end{proof}

\begin{remark}
In principle Propostion~\ref{pro:combinat} gives a combinatorial way of determining the kernel of $\FH_{a,b}$ working with bases and using the straightening algorithm.
The main obstacle is that the straightening algorithm is very slow for tableaux with many boxes.
\end{remark}

\subsection{Evaluating a function}
We fix $p$ many random points $v_1,\ldots,v_p$ in $\Sym^b V$.
To create a basis of $\HWV_\la(\Sym^a\Sym^b V)$
we choose random semistandard tableaux $T$ of shape $\la$ and rectangular content $a \times b$,
take its $(a,b)$-symmetrized tableau $f=\varrho(T)$ (which is an element of $\HWV_\la(\Sym^a\Sym^b V)$) and evaluate $f$ at the $p$ points $v_1,\ldots,v_p$.
We choose more tableaux until we have found $p$ linearly independent tableaux and discard any other tableaux.

We proceed analogously for $\HWV_\la(\Sym^b\Sym^a V)$ with random points $v'_1,\ldots,v'_{p'}$,
but here we store the points $v'_1,\ldots,v'_{p'}$ for later use (as described in Section~\ref{subsec:overview}).

For this approach to work we need a method to evaluate a function given by an $(a,b)$-symmetrized tableau at a point $v \in \Sym^b V$.
Analogously we need a method for the evaluation of $(b,a)$-symmetrized tableaux at a point $v' \in \Sym^a V$.
We will only describe the former one, the latter one being analogous.
We will restrict our description to the case where $v$ is a sum of $a$ powers of linear forms, i.e.,
\[
v = \sum_{i=1}^a (l_i)^b,
\]
where each $l_i \in V$.
We will always choose our points $v_j \in \Sym^b V$ in a way that they satisfy this restriction.
Moreover, we will choose $l_i \in \IZ^{\dim V}$.

Let $T$ be a tableau with rectangular content $a \times b$ and let $f := \varrho(T)$.
The principle of polarization and restitution explains the evaluation $f(v)$ of homogeneous polynomials as a tensor contraction
(where we interpret $f \in \Sym^a\Sym^b V \subseteq \tensor^{ab} V$):
\[
f(v) = \langle f , v^{\otimes a} \rangle.
\]
Since $v^{\otimes a}$ is symmetric under $\aS_a$, we see that
\[
f(v) = \langle T , v^{\otimes a} \rangle.
\]
Expanding $v^{\otimes a}$ we obtain
\begin{equation}\label{eq:largesum}
\langle T , v^{\otimes a} \rangle = \sum_{i_1,i_2,\ldots,i_a = 1}^a \langle T, l_{i_1}^b \otimes l_{i_2}^b \otimes \cdots \otimes l_{i_a}^b\rangle.
\end{equation}
The following claim is the key insight of writing the evaluation in this way.
\begin{claim}
Each summand in \eqref{eq:largesum} is a product of determinants.
\end{claim}
\begin{proof}
Recall that since $T$ has rectangular content $a \times b$ it is a sum of $(b!)^a$ tableaux with rectangular content $(ab)\times 1$.
Note that the group $(\aS_b)^a$ stabilizes $l_{i_1}^b \otimes l_{i_2}^b \otimes \cdots \otimes l_{i_a}^b$.
Therefore it is sufficient to contract
\[
\langle \bar T, l_{i_1}^b \otimes l_{i_2}^b \otimes \cdots \otimes l_{i_a}^b\rangle,
\]
where $\bar T$ is one of the $(ab)\times 1$ summands of $T$, i.e., $\bar T = \pi \zeta_\la$ for some $\pi \in \aS_{ab}$.
 The special form of $\zeta_\la$ (see equation \eqref{eq:def:fla}) implies that a contraction of $\pi\zeta_\la$
 with a decomposable tensor is precisely a product
 of determinants, one determinant for each column of $\la$.
\end{proof}

\section{Speeding up the algorithm}
\label{sec:speedup}
The only problem with the algorithm stated in Section~\ref{sec:naive} is that it is too slow
to do the required calculation in a reasonable amount time.
In this section we present some crucial improvements that make the calculation feasible.

\subsection{The computation tree}
According to Proposition~\ref{pro:combinat} the naive way to compute $\FH_{a,b}^\la$ of an $(a,b)$-symmetrized tableau $f$
is to loop over all ways to place the numbers $1,\ldots,b$ at the positions of the letters in $f$
such that $1,\ldots,b$ gets placed on the A's, $1,\ldots,b$ gets placed on the B's, and so on.
Many of the resulting tableaux will have a column in which two numbers coincide.
By the Grassmann relation these correspond to the zero vector, so they do not contribute to the sum.
We want to avoid enumerating these cases.
A good solution is to place the numbers one after another, each time respecting the constraint
that there should be no column with two coinciding entries.
This can be implemented with a computation tree and depth-first search.
A priori this tree would have $(b!)^a$ leafs, but this constraint prunes the tree considerably.

A second important idea is to use the following common heuristic from computer science:
While we are building the computation tree we are free to choose in which tableau cell we want to place the next number.
We should always choose the cell where we have the fewest possibilities to place numbers.
This heuristic tries to make the decision tree slim.

\subsection{Precaching}
If we do not want to compute $\FH_{a,b}^\la(f)$, but are only interested in
the evaluation $\FH_{a,b}^\la(f)(v)$, then at each computation leaf we compute a product of determinants.
Before starting the algorithm we can quickly compute all the determinants that could possibly arise during the run of this algorithm and store them.
This means that at each leaf of the computation tree we just perform a database lookup operation.

\subsection{Using symmetries}\label{subsec:usingsymmetries}
If two columns $c_1$ and $c_2$ of $f$ coincide, then the set of summands that we get
when applying $\FH_{a,b}(f)$ can be partitioned into two sets:
The ones where the top number in $c_1$ is smaller than the top number in $c_2$
and the ones where the top number in $c_1$ is larger than the top number in $c_2$
(equality cannot hold, since the two cells that we compare have the same letter in $f$).
If we sum up either set of summands, by the Pl\"ucker relations we get the same result.
So we only have to look at one of the two sets.
More generally, if several columns of $f$ coincide, then we can enforce their top entries to be ordered.
Moreover, for each top cell we can determine what its possible minimal and maximal value can be.
For example, when calculating $\FH_{6,6}$
and we have 4 coinciding columns $c_1$, $c_2$, $c_3$, $c_4$, then
the top number of $c_1$ cannot be $4$, $5$, or $6$, as that would mean that the top number of $c_4$ would exceed $6$,
which is impossible since we only place numbers $1,\ldots,6$.
More generally, the (min,max) pairs for 4 coinciding columns are
$(1,3),(2,4),(3,5),(4,6)$ from left to right.
This can readily be generalized to an arbitrary number of columns.
These restrictions help to prune the search tree even earlier.

These symmetry considerations give a large speedup factor if $f$ has many coinciding columns.
It turns out that if the plethysm coefficient is large,
some of the tableaux have only very few coinciding columns.
This is the main reason why we do not study $\ker\FH_{6,6}$ completely.

\subsection{Parallelization}
The computation can be efficiently parallelized on several levels:
First, each type $\lambda$ can be analyzed independently.
Moreover, each tableau $f$ can be evaluated at each point independently.
Such an evaluation of a tableau at a point
can be split up further by fixing a number of processors $N$
and ensuring that in a fixed depth $D$ of the computation tree
the processor with number $C \in \{0,\ldots,N-1\}$ only takes care of the $I$th subtree if $I \equiv C \mod N$.

\subsection{Other speedups}
In principle the algorithm could be sped up further by choosing the random generic points a little less generic,
so that some subsets of $\{l_i \mid 1 \leq i \leq p'\}$ are linearly dependent.
Then we could prune the tree even earlier.
We did not implement this.

It is also possible to not evaluate at a point, but to only contract with a symmetric tensor.
This could speed up the computation if the number of rows is not maximal.
We also did not implement this.

\section{Implementation and running time}
\label{sec:implandrun}
We implemented the above algorithm in C using the following external sources:
\begin{itemize}
\item the BigDigits multiple-precision
arithmetic library Version 2.4 originally written by David Ireland,
copyright (c) 2001-13 D.I. Management Services Pty Limited, all rights reserved,
\item helpful C snippets from stackexchange for string handling,
\item code for ranking and unranking combinations from Derrick Stolee's blog.
\end{itemize}

We ran our code on the Texas A\&M calclab cluster, which provides up to roughly 2000 CPUs for 9 hours each day,
more hours on the weekend.

While $\ker\FH_{5,5}$ can be determined on the cluster in a minute,
the computation of the part of $\ker\FH_{6,6}$ took several days and
the computation for $\ker\FH_{5,6}$ ran for approximately 2 weeks.


\begin{thebibliography}{Kum12}

\bibitem[BHI15]{BHI:15}
Peter B\"urgisser, Jesko H\"uttenhain, and Christian Ikenmeyer.
\newblock Permanent versus determinant: not via saturations.
\newblock arXiv:1501.05528, 2015.

\bibitem[Bri93]{Bri:93}
Michel Brion.
\newblock Stable properties of plethysm: on two conjectures of {F}oulkes.
\newblock {\em Manuscripta Math.}, 80(4):347--371, 1993.

\bibitem[DS00]{DS:00}
Suzie~C. Dent and Johannes Siemons.
\newblock On a conjecture of {F}oulkes.
\newblock {\em J. Algebra}, 226(1):236--249, 2000.

\bibitem[Fou50]{Fou:50}
H.~O. Foulkes.
\newblock Concomitants of the quintic and sextic up to degree four in the
  coefficients of the ground form.
\newblock {\em J. London Math. Soc.}, 25:205--209, 1950.

\bibitem[Ful97]{fult:97}
William Fulton.
\newblock {\em {Y}oung tableaux}, volume~35 of {\em London Mathematical Society
  Student Texts}.
\newblock Cambridge University Press, Cambridge, 1997.

\bibitem[GW09]{gw:09}
Roe Goodman and Nolan~R. Wallach.
\newblock {\em Symmetry, representations, and invariants}, volume 255 of {\em
  Graduate Texts in Mathematics}.
\newblock Springer, Dordrecht, 2009.

\bibitem[Had97]{Had:97}
J.~Hadamard.
\newblock M\'emoire sur l'\'elimination.
\newblock {\em Acta Mathematica}, 20(1):201--238, 1897.

\bibitem[Her54]{Her:1854}
Charles Hermite.
\newblock Sur la th\'eorie des fonctions homogenes \`a deux ind\'etermin\'ees.
\newblock {\em Cambridge and Dublin Mathematical Journal}, 9:172--217, 1854.

\bibitem[How87]{How:87}
Roger Howe.
\newblock {$({\rm GL}_n,{\rm GL}_m)$}-duality and symmetric plethysm.
\newblock {\em Proc. Indian Acad. Sci. Math. Sci.}, 97(1-3):85--109 (1988),
  1987.

\bibitem[Ike12]{ike:12b}
Christian Ikenmeyer.
\newblock {\em Geometric Complexity Theory, Tensor Rank, and
  {L}ittlewood-{R}ichardson Coefficients}.
\newblock PhD thesis, Institute of Mathematics, University of Paderborn, 2012.
\newblock Online available at
  \url{http://math-www.uni-paderborn.de/agpb/work/ikenmeyer_thesis.pdf}.

\bibitem[Kum12]{Kum:12}
Shrawan Kumar.
\newblock A study of the representations supported by the orbit closure of the
  determinant.
\newblock arXiv:1109.5996v2 [math.RT], 2012.

\bibitem[Lan11]{Lan:11}
Joseph Landsberg.
\newblock {\em Tensors: Geometry and Applications}, volume 128 of {\em Graduate
  Studies in Mathematics}.
\newblock American Mathematical Society, Providence, Rhode Island, 2011.

\bibitem[Lan15]{Lan:15}
J.~M. Landsberg.
\newblock Geometric complexity theory: an introduction for geometers.
\newblock {\em Annali dell'universita' di ferrara}, 61(1):65--117, 2015.

\bibitem[Man98]{mani:98}
L.~Manivel.
\newblock Gaussian maps and plethysm.
\newblock In {\em Algebraic geometry ({C}atania, 1993/{B}arcelona, 1994)},
  volume 200 of {\em Lecture Notes in Pure and Appl. Math.}, pages 91--117.
  Dekker, New York, 1998.

\bibitem[McK08]{McK:08}
Tom McKay.
\newblock On plethysm conjectures of {S}tanley and {F}oulkes.
\newblock {\em J. Algebra}, 319(5):2050--2071, 2008.

\bibitem[MN05]{MN:05}
J{\"u}rgen M{\"u}ller and Max Neunh{\"o}ffer.
\newblock Some computations regarding {F}oulkes' conjecture.
\newblock {\em Experiment. Math.}, 14(3):277--283, 2005.

\bibitem[Pro07]{Pro:07}
Claudio Procesi.
\newblock {\em Lie groups}.
\newblock Universitext. Springer, New York, 2007.
\newblock An approach through invariants and representations.

\bibitem[Thr42]{Thr:42}
R.~M. Thrall.
\newblock On symmetrized {K}ronecker powers and the structure of the free {L}ie
  ring.
\newblock {\em Amer. J. Math.}, 64:371--388, 1942.

\end{thebibliography}

\end{document}